\documentclass{amsart}

\usepackage{amssymb,latexsym}
\usepackage{amsfonts}
\usepackage{dsfont}
\usepackage{hyperref}
\usepackage{color}
\usepackage{graphicx}
\usepackage[all]{xy}
\usepackage{ytableau}

\usepackage{caption}
\usepackage{subcaption}

\usepackage{tikz}
\usetikzlibrary{calc, shapes, backgrounds,arrows,positioning,plotmarks}
\tikzset{>=stealth',
  head/.style = {fill = white, text=black},
  plaque/.style = {draw, rectangle, minimum size = 10mm, fill=white}, 
     pil/.style={->,thick},
  junct/.style = {draw,circle,inner sep=0.5pt,outer sep=0pt, fill=black}
  }

\theoremstyle{plain}
\newtheorem{theorem}{Theorem}[section]
\newtheorem{lemma}[theorem]{Lemma}

\newtheorem{examplex}[theorem]{Example}
\newenvironment{example}
  {\pushQED{\qed}\examplex}
  {\popQED\endexamplex}

\theoremstyle{definition}
\newtheorem{remark}[theorem]{Remark}

\numberwithin{equation}{section}

\definecolor{darkblue}{rgb}{0.0,0,0.7}

\newcommand{\newword}[1]{\textcolor{darkblue}{\textbf{\emph{#1}}}}

\newcommand{\bfa}{\mathbf{a}}
\newcommand{\bfb}{\mathbf{b}}

\newcommand{\jdt}{\ensuremath{\mathrm{jdt}}}

\newcommand{\fsl}{\ensuremath{\mathfrak{sl}}}

\newcommand{\sfb}{\ensuremath{\mathsf{b}}}
\newcommand{\sfc}{\ensuremath{\mathsf{c}}}

\newcommand{\evac}{\epsilon}

\begin{document}

\title{Tableau evacuation and webs}

\author{Rebecca Patrias}
\address[RP]{Department of Mathematics, University of St.\ Thomas, St.\ Paul, MN 55105, USA}
\email{rebecca.patrias@stthomas.edu}

\author{Oliver Pechenik}
\address[OP]{Department of Combinatorics \& Optimization, University of Waterloo, Waterloo, ON N2L 3G1, Canada}
\email{oliver.pechenik@uwaterloo.ca}

\date{\today}

\begin{abstract}
Webs are certain planar diagrams embedded in disks. They index and describe bases of tensor products of representations of $\fsl_2$ and $\fsl_3$. There are explicit bijections between webs and certain rectangular tableaux. Work of Petersen--Pylyavskyy--Rhoades (2009) and Russell (2013) shows that these bijections relate web rotation to tableau promotion. We describe the analogous relation between web reflection and tableau evacuation.
 \end{abstract}

\maketitle

\section{Introduction}\label{sec:intro}

G.~Kuperberg \cite{Kuperberg} introduced webs to index and describe bases of tensor products of irreducible representations of low-rank Lie algebras. In this paper, we consider webs for the Lie algebras $\fsl_2$ and $\fsl_3$. Combinatorially, webs are certain planar diagrams embedded in disks. There are explicit bijections between sets of webs and sets of (standard or row-strict) rectangular Young tableaux \cite{Khovanov.Kuperberg,Petersen.Pylyavskyy.Rhoades,Tymoczko,Russell}. It was shown in \cite{Petersen.Pylyavskyy.Rhoades} and \cite{Russell} that these bijections intertwine rotation of web diagrams with the famous but more complicated action called \emph{tableau promotion} (introduced by M.-P.~Sch\"utzenberger \cite{Schutzenberger}). This connection has been useful for establishing combinatorial aspects of promotion through considering the simpler action of web rotation. 

Closely related to promotion is the involution of \emph{tableau evacuation} \cite{Schutzenberger}, which perhaps has been even better studied (see, e.g., \cite{Edelman.Greene,Haiman,Stembridge:evacuation,Stanley:promotion}). It is natural to wonder whether the above bijections carry evacuation to a simple action on webs. In this paper, we show that they do; precisely, evacuation corresponds to reflection of the web diagram. We believe that this fact has been known to experts for some time as folklore, at least in special cases; however, to our knowledge, it has never before been stated in print. The correct general statement of this result is somewhat subtle with respect to exactly how to position the planar diagram and the reflecting line.

Section~\ref{sec:background} reviews necessary background on webs, tableaux, and the bijections in question between them. In Section~\ref{sec:main}, we prove our main result Theorem~\ref{thm:main}. Our proof relies on Lemma~\ref{lem:rotation}, characterizing evacuations of tableaux of all rectangular shapes. This easy lemma is certainly known to experts. However, we are unaware of an explicit proof in the literature, so we include a complete proof of this useful characterization, even though we only need some special cases of it for Theorem~\ref{thm:main}.

\section{Background}\label{sec:background}

\subsection{Webs}
Consider a closed disk in the plane with $2n$ marked points on the boundary. An \newword{$\fsl_2$ web} is a collection of $n$ nonintersecting curves inside the disk, each joining a pair of marked points (considered up to isotopy). In the combinatorics literature, these diagrams also often appear under the name \emph{noncrossing matchings}. Algebraically, the boundary vertices here represent $2$-vectors and each curve corresponds to the determinant of the $2 \times 2$-matrix obtained by concatenating the vectors at its endpoints.

An \newword{$\fsl_3$ web} is similarly a sort of planar bipartite simple graph embedded in a disk (again up to isotopy). These graphs satisfy the following conditions:
\begin{itemize}
\item There are $k$ boundary vertices that are independently colored either black or white.
\item 	The graph is planar and vertices have a fixed bipartite coloring. 
\item Each boundary vertex has degree 1.
\item Each internal vertex has degree 3.
\item Each internal face has at least 6 sides.
\end{itemize}
(Technically, we are only defining ``non-elliptic'' webs here.) Algebraically, here the black vertices represent $3$-vectors, the white vertices represent $3$-covectors, and the edges represent, for example, that a white vertex adjacent to three black vertices $v_1, v_2, v_3$ corresponds to the covector that is the linear operator sending any $3$-vector $v$ to the determinant of the $3 \times 3$ matrix obtained by concatenating $v$ with any two of $v_1, v_2, v_3$. Examples of both sorts of web are given in Figure~\ref{fig:webexamples}.

\begin{figure}
    \centering
    \includegraphics[width=4in]{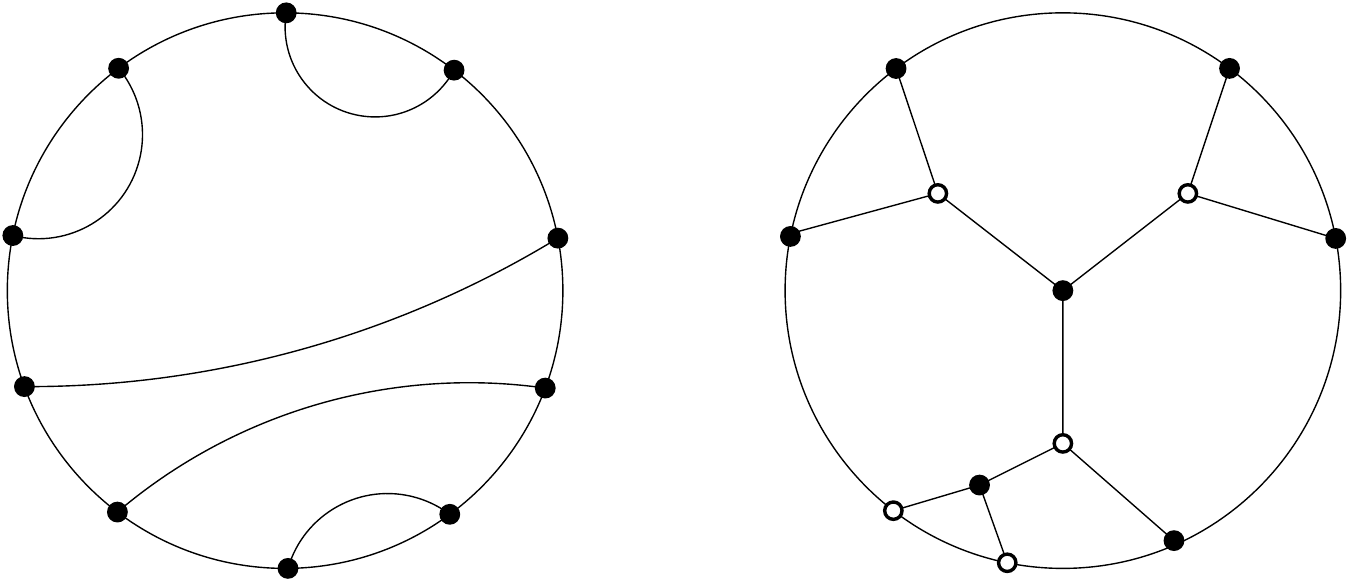}
    \caption{An $\fsl_2$ web (left) and an $\fsl_3$ web (right).}
    \label{fig:webexamples}
\end{figure}

 \subsection{Tableaux and evacuation}

Given an integer partition $\lambda = (\lambda_1 \geq \lambda_2 \geq \dots \geq \lambda_k > 0)$,  the corresponding \newword{Young diagram} is an array of left-justified rows of boxes, with $\lambda_i$ boxes in the $i$th row from the top. We conflate the integer partition $\lambda$ with its Young diagram. We write $\lambda \subseteq \mu$ if the Young diagram of $\lambda$ is a subset of the Young diagram of $\mu$, and we write $\mu / \lambda$ for the \newword{skew Young diagram} that is their set-theoretic difference $\mu \setminus \lambda$. The Young diagram $\lambda$ may be identified with the skew Young diagram $\lambda / \emptyset$, where $\emptyset$ denotes the empty Young diagram. A diagram $\lambda = \lambda/\emptyset$ is sometimes called a \newword{straight shape} (to distinguish it from genuinely skew shapes); a skew Young diagram is called an \newword{anti-straight shape} if it is the $180^\circ$ rotation of a straight shape. In all cases, we write $|\mu / \lambda|$ for the number of boxes of $\mu / \lambda$. 

A \newword{standard Young tableau} of shape $\mu / \lambda$ is a bijective filling of the boxes of that skew Young diagram with the integers $1, \dots, |\mu / \lambda|$ such that entries increase left-to-right across rows and top-to-bottom down columns. In this paper, we will mostly care about the case $\lambda=\emptyset$ and either $\mu=(n,n)$ or $\mu = (n,n,n)$, although we have some need of the general definitions.
 
 A \newword{row-strict tableau} of shape $\mu / \lambda$ is a filling of the boxes of that skew Young diagram with positive integers such that entries strictly increase left-to-right across rows and weakly increase top-to-bottom down columns. If $T$ is a row-strict tableau of shape $\mu = (k,k,k)$ with largest entry $M$ in which each positive integer $1, \dots, M$ appears in $T$ exactly once or twice, we call $T$ a \newword{Russell tableau}. The \newword{repetition} of a Russell tableau $T$ is the number $h$ of values appearing twice. Note that standard Young tableaux of shape $(k,k,k)$ are exactly the Russell tableaux of repetition $0$.
 
 \begin{remark}\label{remark:transpose}
 Instead of row-strict tableaux, it is more common to work with \emph{semistandard tableaux}, i.e., tableaux whose transposes are row-strict. Indeed, H.~Russell \cite{Russell} presents her results using the transposes of what we here call ``Russell tableaux''. However, we prefer to use row-strict tableaux so that the shapes of our tableaux match standard conventions in the web and invariant theory literatures.
 \end{remark}

The \newword{standardization} of a Russell tableau $T$ is defined as follows. Suppose that the value $i$ appears twice in $T$. Replace each entry $j>i$ with $j+1$. One instance of $i$ appears in a strictly lower row than the other. Replace this lower instance with $i+1$. Repeat this process until no value appears more than once in the tableau. The resulting tableau $U$ is a standard Young tableau and is the standardization of $T$. 

\begin{example}
The standardization of the Russell tableau
$\; \ytableaushort{12,13,34} \;$ of repetition $2$ is $\; \ytableaushort{13,24,56}$.	
\end{example}

We will need the following notions of jeu de taquin and rectification for row-strict tableaux. Further details are available in the standard textbooks \cite{Fulton:YoungTableaux, Stanley:EC2, Manivel}, provided the reader is willing to transpose all the conventions (see Remark~\ref{remark:transpose}).
Begin with a skew row-strict tableau $T$ of shape $\lambda/\mu$. Next, choose an empty cell $\sfc$ such that $\sfc$ shares its bottom and/or right edge with $\lambda/\mu$ and $\lambda/\mu\cup \{\sfc\}$ is a skew partition shape. For each such box $\sfc$, we have a tableau $\jdt_\sfc(T)$ that is obtained from the tableau $T$ and is called the \newword{jeu de taquin slide} of $T$ into $\sfc$, which we now describe.

Mark the cell $\sfc$ with $\bullet$. There is at least one box $\sfb_1$ of $\lambda/\mu$ adjacent to $\sfc$. If there are two such boxes and their entries in $T$ are not equal, let $\sfb_1$ denote the box with the smaller entry. If there are two such boxes and their entries in $T$ are equal, let $\sfb_1$ be the box sharing its left edge with $\sfc$. Move the entry in box $\sfb_1$ to cell $\sfc$ and move the $\bullet$ to box $\sfb_1$. Next, consider the cells to the right of and below $\sfb_1$ and repeat this procedure. Continue until the box containing $\bullet$ does not share its right or bottom edge with any box of $\lambda/\mu$, then delete the $\bullet$. The resulting row-strict tableau is $\jdt_\sfc(T)$. 
\begin{example}\label{ex:jdt}
Consider the row-strict tableau $T$ below, where box $\sfc$ is marked with $\bullet$. We compute $\jdt_\sfc(T)$. 
\[T=\begin{ytableau}
\bullet & 1 & 3 \\
1 & 2 & 3\\
3
\end{ytableau}\longrightarrow
\begin{ytableau}
1 & \bullet & 3 \\
1 & 2 & 3\\
3
\end{ytableau}\longrightarrow
\begin{ytableau}
1 & 2 & 3 \\
1 & \bullet & 3 \\
3
\end{ytableau}\longrightarrow
\begin{ytableau}
1 & 2 & 3 \\
1 & 3 & \bullet \\
3
\end{ytableau}\longrightarrow
\begin{ytableau}
1 & 2 & 3 \\
1 & 3  \\
3
\end{ytableau}=\jdt_c(T)\]
\end{example}

If $T$ is a row-strict tableau of skew shape $\lambda / \mu$ with $\mu \neq \emptyset$, we may successively perform jeu de taquin slides $|\mu|$ times until we obtain a row-strict tableau of some partition shape $\xi \subset \lambda$. While this resulting tableau appears to depend on the order in which cells are chosen for jeu de taquin slides, it is in fact uniquely determined by $T$ and is called the \newword{rectification} of $T$.

Using these notions of jeu de taquin, we are now ready to define the main operation considered in this note.
Let $T$ be a row-strict tableau of shape $\lambda$ and largest entry $n$. Suppose $T$ contains $z$ boxes with entry $1$. These boxes necessarily appear as the top $z$ boxes in the leftmost column of $\lambda$. 
Index these boxes as $\sfb_1, \sfb_2, \dots, \sfb_z$ from top to bottom. Let $\tilde T$ be the skew row-strict tableau obtained from $T$ by deleting each box $\sfb_i$ (and its entry) and subtracting $1$ from the entry in each of the other boxes of $T$. Then define $\Delta(T)=\jdt_{\sfb_1}\left( \jdt_{\sfb_2}(\cdots \jdt_{\sfb_z}(\tilde T) \dots )\right)$.

Let $\lambda_i$ denote the shape of the $i$th iterate $\Delta^i(T)$ and note that $\lambda_{i+1} \subset \lambda_i$. The \newword{evacuation} of $T$, denoted $\epsilon(T)$, is the row-strict tableau of shape $\lambda$ with entry $n+1-i$ in boxes $\lambda^i/\lambda^{i-1}$.  

\begin{example}\label{ex:evacuation}
Consider tableau $T$ below with largest entry 5. We compute $\Delta^i(T)$ for $i=1,2,3,4$ below and note that $\Delta^5(T)=\emptyset$. 
\[
T= \begin{ytableau}
1 & 3 & 4\\
2 & 3  \\
4 & 5
\end{ytableau}\hspace{.15in}
\Delta(T)= \begin{ytableau}
1 & 2 & 3\\
2 & 4  \\
3 
\end{ytableau}\hspace{.15in}
\Delta^2(T)= \begin{ytableau}
1 & 2\\
1 & 3  \\
2 
\end{ytableau}\hspace{.15in}
\Delta^3(T)= \begin{ytableau}
1 & 2\\
1   
\end{ytableau}\hspace{.15in}
\Delta^4(T)= \begin{ytableau}
1 
\end{ytableau}\]
Using these computations, we see that 
\[\epsilon(T)=
\begin{ytableau}
1 & 2 & 4\\
2 & 3 \\
3 & 5
\end{ytableau}\ .\]
\end{example}

Additional background on the evacuation operator and its properties may be found, for example, in \cite{Haiman,Stanley:EC2,Stanley:promotion,Bloom.Pechenik.Saracino}.

\subsection{Bijections}
We now describe explicit bijections between various sets of tableaux and corresponding sets of webs with cyclically numbered boundary vertices. At the end of this subsection, we provide Example~\ref{ex:bijections} to clarify all these bijections.

First, we give a bijection between $2$-row standard Young tableaux and $\fsl_2$ webs with cyclically numbered boundary vertices. Let $T$ be a $2\times n$ rectangular standard Young tableau. 
Find some entry $t_1$ in the top row such that $b_1=t_1+1$ is in the bottom row, and match them. Then find some $t_2 <b_2$ such that $t_2$ and $b_2$ are consecutive among unpaired values, $t_2$ appears in the top row, and $b_2$ appears in the bottom row; pair them. Repeat this process, pairing some $t_3 < b_3$ such that they are consecutive among unpaired values and have $t_3$ in the top row but $b_3$ in the bottom row, etc. After $n$ iterations, all values of $T$ are paired.

\begin{remark}\label{rem:pairing_order}
The pairing produced above is independent of the order in which we pair. Traditionally, this pairing is described so that $b_1, b_2, \dots, b_n$ are the labels of the bottom row of $T$ from left to right. Another natural choice would be so that $t_1, t_2, \dots, t_n$ are the labels of the top row of $T$ from \emph{right to left}. We will use sometimes the traditional order and sometimes this other natural one, as convenient.
\end{remark}

 Now, cyclically label $2n$ marked boundary points on a disk, and draw an arc connecting point $i$ and point $j$ whenever $i$ and $j$ are paired in the tableau $T$ as above. The image of this map is all $\fsl_2$ webs with cyclically numbered boundary vertices. We will refer to this bijection as the \newword{Catalan bijection} since both sets are enumerated by \emph{Catalan numbers} \cite{Stanley:Catalan}. 
 
We next discuss an explicit bijection between 3-row standard Young tableaux and $\fsl_3$ webs with all black boundary vertices and with a cyclical numbering on boundary vertices. This bijection was first introduced by M.~Khovanov--G.~Kuperberg \cite{Khovanov.Kuperberg} as a series of recursive growth rules and later simplified by K.~Petersen--P.~Pylyavskyy--B.~Rhoades \cite{Petersen.Pylyavskyy.Rhoades}. We follow a more explicit description of the bijection due to J.~Tymoczko \cite{Tymoczko}.

Let $T$ be a $3 \times k$ rectangular standard Young tableau. Pair the entries of the top two rows of $T$ as in the Catalan bijection described above. Similarly, pair the entries of the bottom two rows of $T$ as in the Catalan bijection. Observe that each entry of the middle row is thereby paired with one entry of each other row, so the entries of $T$ are partitioned into $k$ blocks of size $3$. Cyclically label $3k$ black boundary vertices on a disk, and draw arcs connecting point $i$ to points $j_1, j_2$ whenever $i$ is a label of the middle row of $T$ that shares its block with $j_1$ and $j_2$. We call this diagram an \newword{$m$-diagram}. Replace each degree $2$ boundary vertex by a $Y$-shaped subdiagram (introducing a new white vertex):
\begin{equation}\label{eq:tripods}
\includegraphics[width=4in]{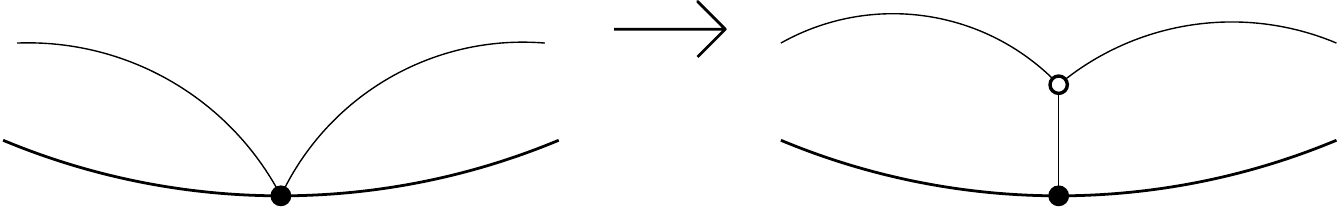}
\end{equation}
The diagram obtained may not be planar; draw it so as to minimize crossings. Now, resolve remaining crossings according to the following local rule:

\begin{equation}\label{eq:resolve_crossing}
    \includegraphics[width=3in]{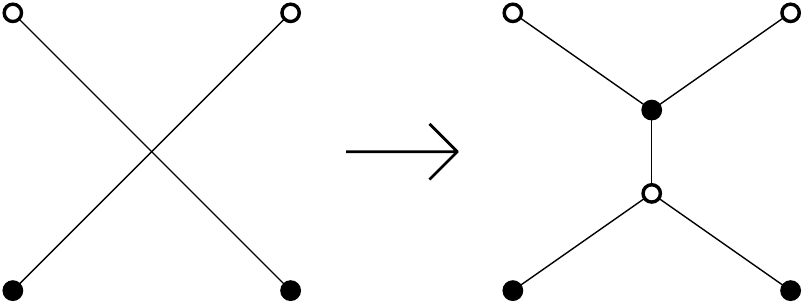}
\end{equation}
The result is an $\fsl_3$ web with all boundary vertices black and cyclically numbered. We call this map the \newword{Tymoczko bijection}.

Lastly, we describe an explicit bijection \cite{Russell} between $3\times k$ Russell tableaux with repetition $h$ and $\fsl_3$ webs with $h$ white boundary vertices and $3k-2h$ black boundary vertices. Let $T$ be a Russell tableau, and let $U$ be its standardization. Use the Tymoczko bijection to build an $\fsl_3$ web from $U$; call it $W_U$. Suppose value $i$ appears twice in $T$, and the corresponding cells of $U$ contain values $j$ and $j+1$. Then perform the following local operation on $W_U$ at boundary vertices $j$ and $j+1$, deleting both boundary vertices and moving their neighboring white vertex to the boundary:
\begin{equation}\label{eq:Russell_reduction}
	    \includegraphics[width=4in]{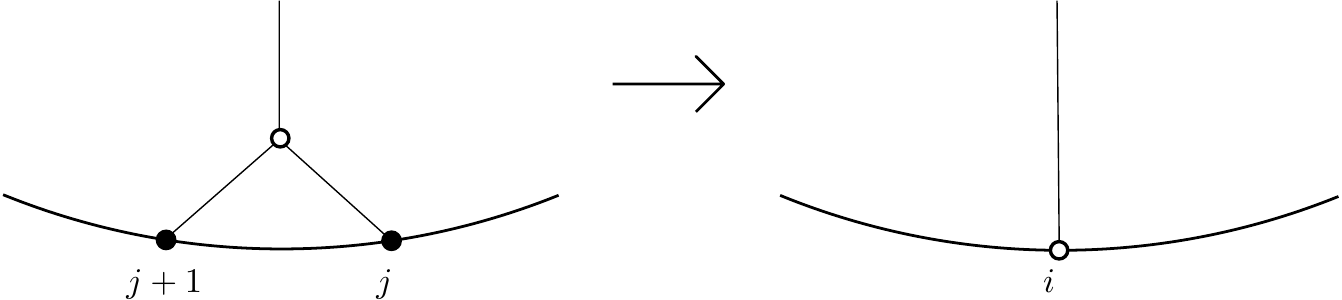}
\end{equation}
The resulting diagram is the $\fsl_3$ web for $T$. Note that this \newword{Russell bijection} restricts to the Tymoczko bijection for standard Young tableaux (i.e., Russell tableaux of repetition $0$).

\begin{example}\label{ex:bijections}
Let 	
$T= \begin{ytableau}
1 & 2 & 3 \\
1 & 4 & 5 \\
3 & 6 & 7
\end{ytableau}$
be a Russell tableau of shape $(3,3,3)$ and repetition $2$. We show how to produce a corresponding $\fsl_3$ web with $2$ white boundary vertices and $3 \cdot 3 - 2 \cdot 2 = 5$ black boundary vertices. This example is intended to also clarify the Catalan and Tymoczko bijections.

Observe that the standardization of $T$ is
$U= \begin{ytableau}
1 & 3 & 4 \\
2 & 6 & 7 \\
5 & 8 & 9 
\end{ytableau}$. In the top two rows of $U$, we pair $2$--$1$, $6$--$4$, and $7$--$3$ as in the Catalan bijection. In the bottom two rows of $U$, we pair $5$--$2$, $8$--$7$, and $9$--$6$. These pairings yield the $m$-diagram shown at top left below. Next, we replace degree $2$ boundary vertices with tripods to obtain the diagram in the middle of the top row below. This diagram is nonplanar, so we resolve the three crossings, obtaining the diagram on the right, which is the web corresponding to $U$ under the Tymoczko bijection.

Finally, we recall that the repeated values of $T$ are $1$ and $3$, corresponding to the values $1, 2$ and $4,5$ in $U$. Hence, we contract the web at these vertices to obtain the diagram shown in the second row below, the web corresponding to $T$.
\begin{center}
    \includegraphics[width=5in]{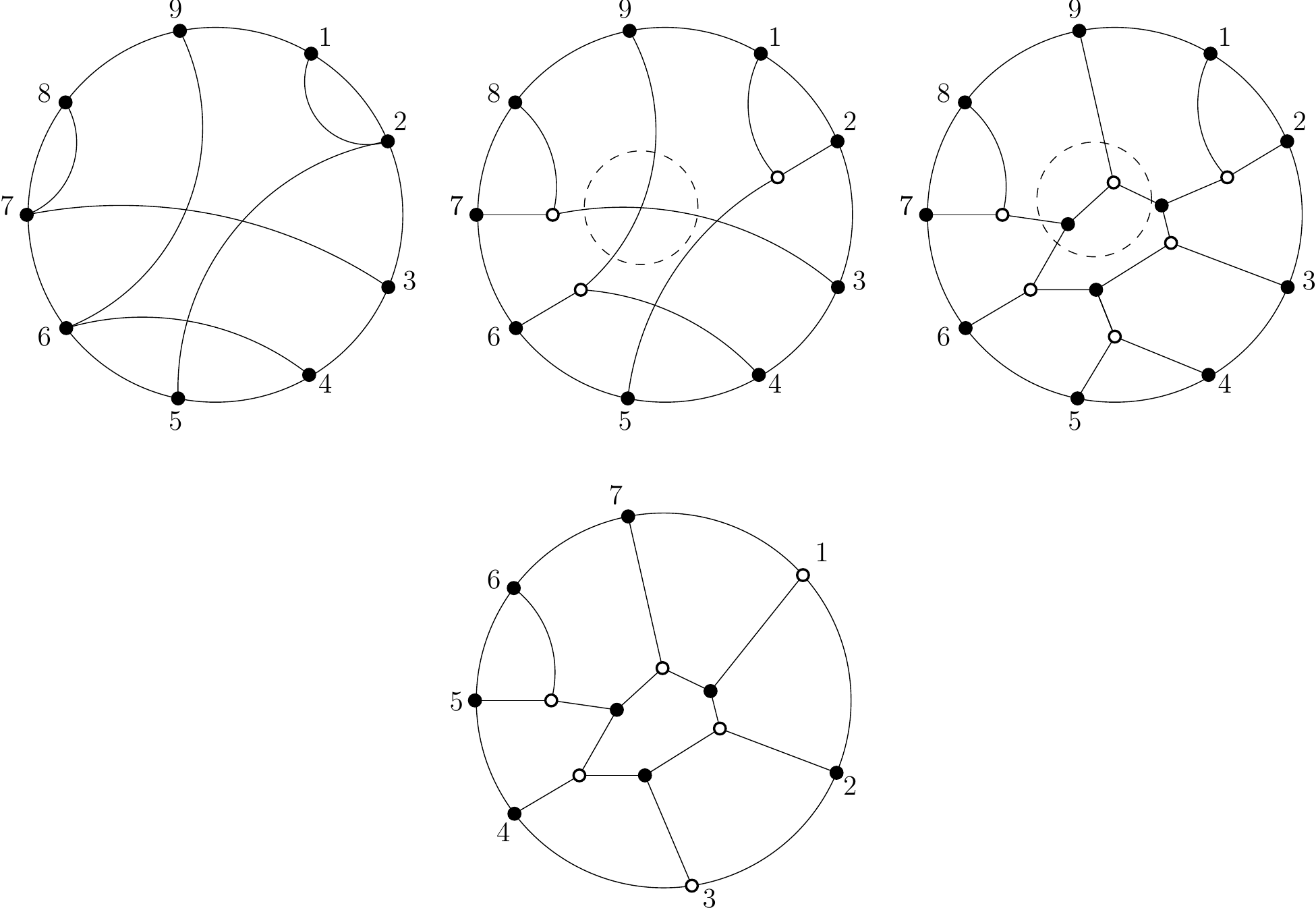}
\end{center}
\end{example}

Given an $\fsl_2$ web $W$ with cyclically numbered boundary vertices, we write $\mathbb{T}(W)$ for the corresponding tableau under the Catalan bijection. Similarly, if $W$ is an $\fsl_3$ web with cyclically numbered boundary vertices, we write $\mathbb{T}(W)$ for the corresponding tableau under the Russell bijection. Conversely, if $T$ is a tableau of appropriate type, we write $\mathbb{W}(T)$ for the corresponding web. Note that $\mathbb{T}$ and $\mathbb{W}$ are inverses.

A different bijection for Russell tableaux was given in \cite{Benkart.Cho.Dongho}. A related bijection between webs and \emph{oscillating tableaux} was described in \cite{Patrias:oscillating}.
For some other discussion of relations between tableau bases and web bases (mostly in the $\fsl_2$ case), see \cite{Russell.Tymoczko,Rhoades:webs}. In \cite{Hopkins.Rubey}, a relation was described between certain $\fsl_3$ webs and promotion and evacuation of linear extensions of a ``Kreweras poset'' $V(n)$; we do not understand how to connect this relation to the observations of the current paper, although they are clearly of a similar flavor.

\section{Main result}\label{sec:main}

For general shapes, evacuation appears complicated, and yet it is non-obviously an involution. However, for rectangular shapes, evacuation has the following much simpler characterization.
We believe the following easy but useful lemma is well-known. However, we are unaware of a proof appearing in the literature so we include a complete one below, although we will only need the result in the cases that $T$ is a $2$-row standard Young tableau or a $3$-row Russell tableau.

\begin{lemma}\label{lem:rotation}
	Let $T$ be a row-strict tableau of rectangular shape $\lambda = (k,k, \dots, k)$ with maximum label $n$. Then $\evac(T)$ is given by rotating $T$ by $180^\circ$ and reversing the alphabet so that $x \mapsto n+1-x$.
\end{lemma}

Before we prove this lemma, we need to recall a few definitions and standard facts.
The \newword{reading word} of a (skew) tableau $T$ is the word obtained by reading the entries of $T$ down columns moving right to left. For example, the reading word of $\; \ytableaushort{12,13,34} \;$ is $234113$.

Given a word $w$, we define its \newword{Greene--Kleitman invariants} \cite{Greene:Schensted,Greene.Kleitman} as follows. Say $u$ is a \newword{$d_i$-subword} of $w$ if $u$ can be decomposed into $i$ pairwise disjoint nondecreasing subwords. Let $\Delta_i(w)$ be the length of the longest $d_i$-subword of $w$. Note that $\Delta_0(w) = 0$. An important fact is that, if $w$ is the reading word of a row-strict tableau $T$ of straight shape $\lambda$, then column $i$ of $\lambda$ has $\Delta_i(w) - \Delta_{i-1}(w)$ boxes. Another well-known fact we will need is that the Greene--Kleitman invariants of the reading word of a tableau are preserved under performing jeu de taquin slides. For a textbook exposition of these facts, see \cite[Chapter 3]{Fulton:YoungTableaux} (where the Greene-Kleitman invariant $\Delta_i(w)$ is denoted $L(w,i)$ and not given a name).

\begin{proof}[Proof of Lemma~\ref{lem:rotation}]
	For $U$ a tableau, let $U_{\leq a}$ denote the subtableau of $U$ labeled by entries that are at most $a$, and define $U_{>b}$ similarly. 
	
	By definition of evacuation, the shape of $\evac(T)_{\leq a}$ is the shape of the rectification of $T_{>n-a}$. Hence, the Greene--Kleitman invariants of $\evac(T)_{\leq a}$ and $T_{>n-a}$ are the same. Since $\evac(T)_{\leq a}$ is of straight shape and  $T_{>n-a}$ is of anti-straight shape, it follows that the shapes of these tableaux are $180^\circ$ rotations of each other. The lemma follows.
\end{proof}

\begin{example}
Beginning with the row-strict tableau $T$ with largest entry 8, we first rotate $T$ by $180^\circ$, and then reverse the alphabet by replacing each entry $x$ with $8+1-x$. The resulting tableau is $\epsilon(T)$: 
\begin{center}
\raisebox{-.35cm}{$T=$} \begin{ytableau}
1 & 2 & 3 & 5 \\
1 & 2 & 4 & 6\\
3 & 5 & 7 & 8
\end{ytableau} \raisebox{-.35cm}{$\longrightarrow$}
\begin{ytableau}
8 & 7 & 5 & 3 \\
6 & 4 & 2 & 1 \\
5 & 3 & 2 & 1 
\end{ytableau} \raisebox{-.35cm}{$\longrightarrow$}
\begin{ytableau}
1 & 2 & 4 & 6  \\
3 & 5 & 7 & 8  \\
4 & 6 & 7 & 8
\end{ytableau} \raisebox{-.35cm}{$=\epsilon(T)$.}
\end{center}
We may also note that the reading word of $T$ is $w=568347225113$. Since $T$ is of straight shape with each column consisting of $3$ boxes, it is guaranteed that the Greene--Kleitman invariants will be $(\Delta_1(w),\Delta_2(w),\Delta_3(w),\Delta_4(w)) = (3,6,9,12)$. The reader may check, for example, that $568347$ is a subword of $w$ with length $6$ that decomposes into $2$ disjoint nondecreasing subwords, but that there is no longer such subword of $w$ (although there are others of the same length).
\end{example}

We next define the type of web reflection we need for our main result. To this end, let $W$ be a web with boundary vertices cyclically numbered $1, 2, \dots, j$. If any boundary vertices are white, temporarily replace them with pairs of black boundary vertices by reversing the Russell contraction~\eqref{eq:Russell_reduction}. If boundary vertex $i\neq j$ was white, name the new black vertices $i$ and $i+0.1$ maintaining cyclic order. If boundary vertex $j$ was white, name its new black vertices $j-0.1$ and $j$, again maintaining cyclic order. Now, slide the vertices along the boundary until they are evenly spaced. 

Let $\bfa$ be the midpoint of the boundary arc of the disk between vertices $j$ and $1$.
Next, perform the Russell contraction~\eqref{eq:Russell_reduction} for each pair of boundary vertices previously obtained by reversing that contraction, placing each white boundary vertex at the midpoint of the boundary arc joining its black parent vertices. We have now constructed a web that is equivalent to the original one, but with more convenient spacing. In particular, we have arranged for white vertices to be spaced further apart than black vertices, as in the final diagram of Example~\ref{ex:bijections}.

Let $\bfb$ be the antipodal point to $\bfa$ and let $\mathcal{D}$ be the diameter joining $\bfa$ to $\bfb$. (Depending on the number of white vertices and the parity of $j$, $\bfb$ may either coincide with one of the boundary vertices of $W$ or not.)

\begin{theorem}\label{thm:main}
	Let $W$ be a web with cyclically numbered boundary vertices.  Let $W'$ be the reflection of $W$ across the diameter $\mathcal{D}$ (as constructed above). Then 
	\[
	\mathbb{T}(W') = \evac(\mathbb{T}(W)).
	\]\
\end{theorem}
\begin{proof}
	Let $U = \mathbb{T}(W)$ and let $U' = \evac(U)$. We will show that $U' = \mathbb{T}(W')$. We first consider the case that these are standard Young tableaux.
	
	By Lemma~\ref{lem:rotation}, $U'$ is obtained from $U$ by rotating $180^\circ$ and reversing the alphabet. If $U$ has $2$ rows, then the pairing for $U$ of the Catalan bijection is the $180^\circ$ rotation of the pairing for $U'$, since pairing $U$ according to the traditional order described in Remark~\ref{rem:pairing_order} corresponds to pairing $U'$ according to the other order described in that remark. Hence, the webs for $U$ and $U'$ are reflections of each other, as desired.
	
	Suppose instead that $U$ has $3$ rows. 
	It suffices to show that the $m$-diagram for $U'$ is the reflection of the $m$-diagram for $U$, since the rules \eqref{eq:tripods} and \eqref{eq:resolve_crossing} for obtaining webs from $m$-diagrams are clearly reflection-symmetric.
	Moreover, $m$-diagrams only depend on relations between pairs of consecutive tableau rows. 
	 Hence, the theorem for $3$-row standard Young tableaux then follows from Lemma~\ref{lem:rotation} and the above symmetry of the pairings in the $2$-row case.
	
	Finally, suppose that $U$ is a Russell tableau. Let $V$ and $V'$ be the standardizations of $U$ and $U'$, respectively. We have already shown that $\mathbb{W}(V)$ and $\mathbb{W}(V')$ are reflections of each other. It then follows from Lemma~\ref{lem:rotation} that we obtain $\mathbb{W}(U)$ and $\mathbb{W}(U')$ from $\mathbb{W}(V)$ and $\mathbb{W}(V')$ by applying the Russell contraction \eqref{eq:Russell_reduction} in positions that differ by reflection across $\mathcal{D}$. Since the Russell contraction rule \eqref{eq:Russell_reduction} is also reflection-symmetric, the theorem then follows.
\end{proof}

\begin{example}
Consider the Russell tableau $T$ and corresponding $\fsl_3$ web $W$ from Example~\ref{ex:bijections}. The points $\bfa$ and $\bfb$ are shown on $W$ below. Reflecting over the dotted line $\mathcal{D}$ connecting $\bfa$ and $\bfb$ gives the $\fsl_3$ web $W'$, which the reader can verify is the web corresponding to $\epsilon(T)$.
\begin{center}
$T=\begin{ytableau}
1 & 2 & 3 \\ 1 & 4 & 5 \\ 3 & 6 & 7
\end{ytableau}$\hspace{1.3in}
$\epsilon(T)=\begin{ytableau}
1 & 2 & 5\\
3 & 4 & 7 \\
5 & 6 & 7
\end{ytableau}$\\

\vspace{0.7cm}
\includegraphics[width=5in]{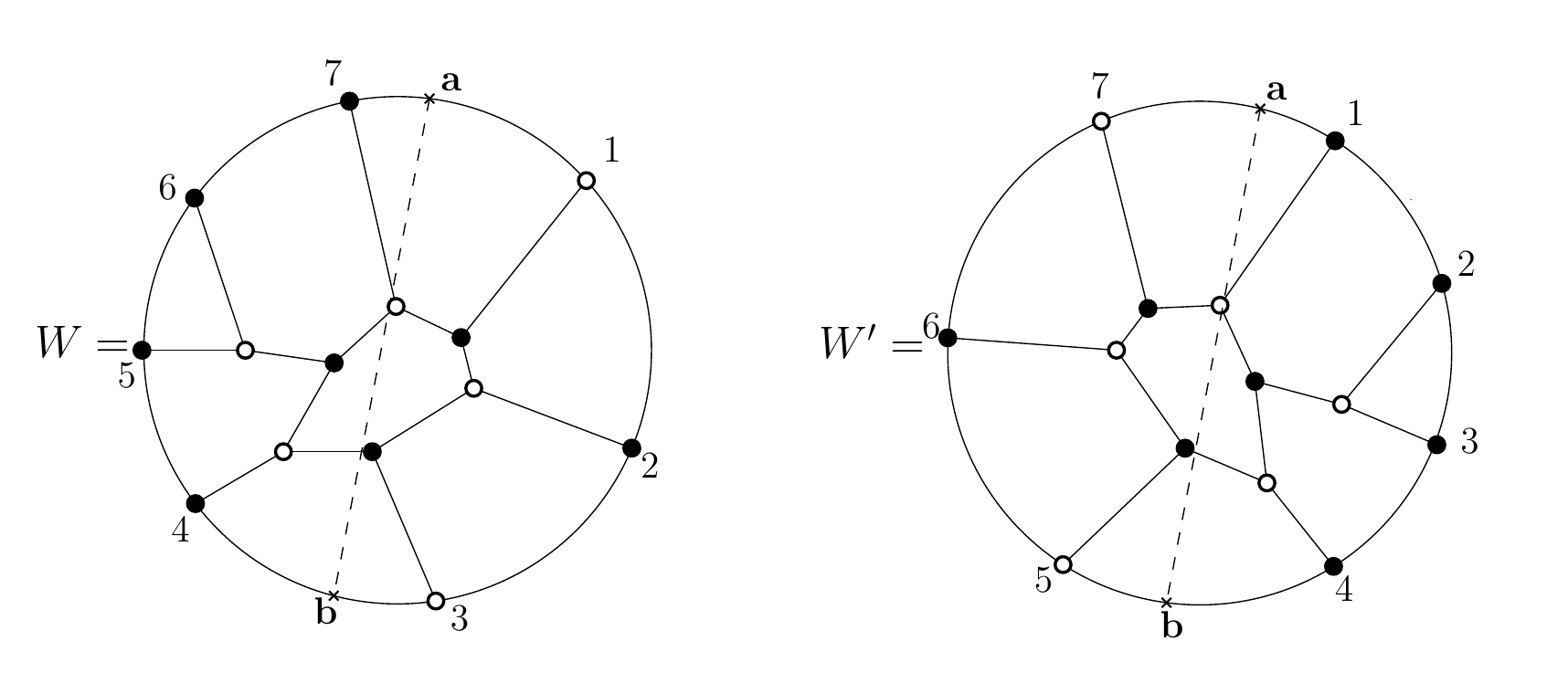}
\end{center}
\end{example}

\section*{Acknowledgements}
We would like to thank Pasha Pylyavskyy for encouraging us to write this paper.
Thanks also to Sam Hopkins, Brendon Rhoades, and Jessica Striker for very helpful conversations and comments on an earlier draft of this manuscript.

OP acknowledges support from NSERC Discovery Grant RGPIN-2021-02391 and Launch Supplement DGECR-2021-00010.

\bibliographystyle{amsalpha} 
\bibliography{evacuation}

\end{document}